\documentclass[12pt, letterpaper]{article}
\usepackage[utf8]{inputenc}
\usepackage{amsmath,amssymb,amsfonts,amsthm}
\usepackage{bm}
\usepackage[margin=2cm]{geometry}
\usepackage{graphicx} 
\usepackage{xcolor}
\usepackage{url}
\usepackage{mathtools}
\usepackage{todonotes}
\newtheorem{theorem}{Theorem}[section]

\newtheorem{remark}{Remark}
\title{Optimal Control of Traveler Testing in Epidemic Models}
\author{Maria Teresa Chiri$^1$
	\and Christopher Denaro$^2$
    \and Xiaoqian Gong$^3$
    \and Benedetto Piccoli$^4$
	}

\date{%
	\small\textit{$^1$Department of Mathematics and Statistics, Queen's University, Kingston, ON, Canada\\	(maria.chiri@queensu.ca)\\%
	$^2$Department of Mathematical Sciences and Center for Computational and Integrative Biology, Rutgers University, Camden, NJ, USA, \\ (cad373@scarletmail.rutgers.edu)}\\
   $^3$ Department of Mathematics, Amherst College Amherst, MA, USA,\\
(xgong@amherst.edu) \\
$^4$Department of Mathematical Sciences and Center for Computational and Integrative Biology, Rutgers University, Camden, NJ, USA,\\ (piccoli@camden.rutgers.edu)\\[2ex]%
	\today
}

\begin{document}

\maketitle
\begin{abstract}
    We propose an optimal control framework for mitigating cross-border epidemic spread through testing of incoming travelers. The general model considers multiple countries with SEIR dynamics, while the explicit analysis focuses on a simplified two-country SIR system without vaccination. The control variable is a time-measurable function representing the testing rate. For the two-country SIR case, we show the optimal control is bang–bang by excluding singular arcs via Hamiltonian analysis and derive the optimal switching time. Our results offer theoretical insights and practical guidelines for designing cost-effective border testing policies. We then present numerical simulations for the optimal control in the two-population setting.
\end{abstract}
\section{Introduction}
The ongoing challenge of controlling infectious diseases has motivated extensive research into optimal intervention strategies within compartmental epidemic models such as the Susceptible-Infectious-Recovered (SIR)
and Susceptible-Exposed-Infected-Recovered (SEIR). 
A broad body of work has analyzed time-optimal control problems, revealing that many effective policies exhibit bang–bang characteristics—switching sharply from no intervention to maximal effort—to balance epidemic duration and infection burden \cite{balderrama2022optimal,bliman2021optimal,bolzoni2017time}. Importantly, these studies demonstrate a critical trade-off: minimizing outbreak length may conflict with reducing the total number of cases, sometimes leading to delayed intervention even when early control might intuitively seem preferable \cite{bolzoni2017time,bolzoni2014react}.

Extensions to these classical models incorporate practical considerations such as finite intervention windows, resource constraints, and economic costs. For example, limited-duration quarantines and social distancing measures have been optimized to minimize long-term infections while balancing socioeconomic impacts \cite{balderrama2022optimal,bliman2021optimal}. Other investigations explicitly consider costs, finding that optimal strategies may combine vaccination, isolation, or culling efforts in nontrivial ways depending on resource availability and disease dynamics \cite{ketcheson,behncke2000optimal,hansen2011optimal,bolzoni2014react}. Time-optimal policies under limited control budgets or mixed interventions (e.g., vaccination and isolation) exhibit rich structures and often require careful coordination between strategy onset and duration \cite{bolzoni2014react,hansen2011optimal}.
The COVID-19 pandemic gave rise to multiple studies motivated by minimizing death to the population \cite{rowan2021disposable}, easing out of the non-pharmaceutical interventions \cite{vardavas2021modeling}, studying the role of behavior in transmission \cite{allred2022covid}, and modeling of longer time-spans for decision makers \cite{badoual2021vaccination, celani2022endemic}.  
Moreover, several papers considered specific control problems, including
modeling the containment measures \cite{Gatto2020}, discussing the controllability of the pandemic in the early phase \cite{casella2020can}, 
and timing the interventions
\cite{aronna2021model,gevertz2021novel,perkinsoptimal}.

Beyond single-region models, recent efforts have focused on multi-region epidemic spread interconnected by travel, emphasizing the importance of coordinated policy responses across countries or regions. Hybrid dynamical systems and automata-based frameworks have been employed to model discrete travel restrictions alongside continuous disease dynamics, capturing the timing and triggering of travel bans based on epidemiological thresholds \cite{carney2024using}. Additional layers of complexity arise when behavioral responses or multiple interventions are integrated into epidemic models. For example, evolutionary game-theoretic approaches have been coupled with SIR dynamics to study fast switching in vaccination strategies \cite{della2024geometric}, while extended SEIR models have been developed to evaluate combinations of social distancing, testing, contact tracing, and quarantining under economic constraints \cite{mcquade}. Such hybrid and multiscale approaches offer a more realistic representation of both policy implementation and public response, enabling the design of interventions that dynamically adapt to both epidemiological and socioeconomic conditions.

In this work, building on the modeling framework proposed in \cite{carney2024using}, we introduce a multi-country SIR system where countries are connected by travel flows, and the control variable is a time-measurable testing rate for incoming travelers.
While the formulation accommodates an arbitrary number of countries and explicitly includes vaccinated and unvaccinated compartments, we perform a detailed qualitative study for a simplified two-country SIR model without vaccination.
In this reduced setting, analysis of the associated Hamiltonian allows us to rule out the existence of singular controls.
While the control variable is allowed to be any measurable function of time, the analysis of the Hamiltonian reveals that the optimal solution necessarily takes the form of a bang–bang control. This structural property introduces discrete switching times in an otherwise continuous system, thereby endowing the overall dynamics with a hybrid character.
We then determine the optimal switching time analytically, identifying the precise moment at which testing should shift from full deployment to suspension (or vice versa).
This combination of a general multi-country framework and an explicit analytical characterization in a simplified setting provides both theoretical insight and practical guidance for the design of cost-effective border-screening strategies.\\

The remainder of the paper is organized as follows.
In Section \ref{S2}, we introduce a hierarchy of epidemic models with increasing complexity. We begin with the SIR model without vaccination, then extend it to the SEIR model without vaccination, and finally to the SEIR model with vaccination, thereby establishing a general framework capable of capturing different epidemiological and intervention settings.
In Section \ref{sec:optimal_control_problems}, we formulate and analyze the optimal control problem for the two-country SIR model without vaccination, where the control variable represents the testing rate applied to incoming travelers. The admissible controls are assumed to be measurable functions of time, and through quantitative analysis we demonstrate that the optimal policy must be of bang–bang type. This structure motivates restricting attention to piecewise constant controls, corresponding to distinct phases of intervention. In particular, we examine switching patterns such as transitions
from 0 to 1 (representing a shift from no testing to maximal testing),
from 1 to 0 (the suspension of testing), and sequences such as 0–1–0, representing temporary testing campaigns. For each configuration, we analytically determine the corresponding optimal switching times.
Finally, Section \ref{S4} presents numerical simulations illustrating the analytical results and highlighting how the timing and intensity of border-testing interventions influence epidemic trajectories and inter-country transmission dynamics.

\section{The general model}\label{S2}
\subsection{Hybrid SIR Model Without Vaccination}

We introduce a Susceptible--Infected--Removed (SIR) model accounting for inter-regional mobility and localized testing policies across \(N\) regions. Within each region, the population is divided into three compartments: \( S \) (Susceptible), consisting of individuals at risk of infection; \( I \) (Infected), representing symptomatically infected individuals who can transmit the disease; and \( R \) (Removed), which includes those who have either recovered or died from the disease.

The progression of disease within each region is governed by differential equations using two key parameters: the transmission rate \( \beta_i \), indicating the likelihood of transmission of disease in the region \( i \), and the recovery rate \( \gamma_i \), which represents the rate at which infected individuals recover. Simultaneously, regional testing policies are encoded through control variables \( u_i \in [0, 1]\), where \( u_i\) corresponds to the implementation of the testing rate for incoming individuals in region \( i \). That is, \(u_i=0\) indicates no testing of incoming individuals, while \(u_i=1\) corresponds to full testing implementation in region \(i\). 

Moreover, spatial coupling is achieved by incorporating inter-regional travel, further reinforcing the hybrid nature of the system. In addition to resident movement, a key innovation of this model is the explicit inclusion of visitors, who temporarily enter other regions and interact with local populations before returning to their home regions. Infectious visitors can seed new infections in the destination region, impacting local outbreaks. Although their stay is transient, they can significantly influence the spatial spread of the disease. Upon returning home, visitors may bring back new infections acquired while traveling, creating a feedback loop that links epidemic dynamics across regions. Mobility is modeled through travel weights: \( w_{ji}^{I} \) captures the rate of infected individuals moving from region \( j \) to region \( i \), and \( w_{ij}^{S} \) denotes the rate of susceptible individuals traveling from region \( i \) to region \( j \). Noting that 
\begin{align}
\label{eqn: travel_weights}
   \sum\limits_{j=1}^{N} \omega_{ij}^{S} =1, \text{ and } \sum\limits_{j=1}^{N} \omega_{ij}^{I} =1. 
\end{align}
We define the vector of infected populations across all regions as \( I = [I_1, I_2, \dots, I_N] \), and the vector of travel weights corresponding to the inflow of infected individuals into region \(i\) as
\[
W^{I}_{\cdot i} = [w_{1i}^{I}, w_{2i}^{I}, \dots, w_{Ni}^{I}]^{T} \in \mathbb{R}^{N}.
\]
The total inflow of infected individuals into region \(i\) is then given by the weighted sum of infected populations in all regions \(k = 1, \dots, N\), where each term is scaled by the corresponding travel rate to region \(i\). This is expressed as the inner product:
\[
\langle W^{I}_{\cdot i}, I \rangle = \sum\limits_{j=1}^{N} w_{ji}^{I} I_j.
\]


The dynamics of \(S_i, I_i, R_i\) populations in region \(i \in \{1, 2, \dots, N\}\) are given by the following:  
\begin{equation}
\begin{aligned} 
    & \frac{d}{dt} S_i(t) =- S_i(t) w^{S}_{ii}
    \beta_i (1-u_i)\sum\limits_{j \not = i} \omega_{ji}^I I_j(t)
     - S_i(t) \sum \limits_{j\not = i}w_{ij}^{S}
    \beta_jw^{I}_{jj}I_j(t) \\&\qquad \qquad \qquad-S_i(t) \sum \limits_{j\not = i}w_{ij}^{S}
    \beta_j\sum\limits_{k \not = j} (1-u_j)w_{kj}^I I_k - S_i(t)w_{ii}^S \beta_i w_{ii}^I I_i(t),\\
  & \frac{d}{dt} I_i(t) =S_i(t) w^{S}_{ii}
    \beta_i (1-u_i)
    \sum\limits_{j \not = i} w^{I}_{ji} I_{j}(t) + S_i(t)\sum\limits_{j \not = i}  w_{ij}^{S}
    \beta_jw^{I}_{jj}I_j(t) \\& \qquad \qquad \qquad+ S_i(t) \sum \limits_{j\not = i}w_{ij}^{S}
    \beta_j\sum\limits_{k \not = j} (1-u_j)w_{kj}^I I_k+ S_i(t)w_{ii}^S \beta_i w_{ii}^I I_i(t) -\gamma_i I_i(t), \\
      & \frac{d}{dt} R_i(t) = \gamma_iI_i(t),
\end{aligned}
\label{eqn: SIR}
\end{equation}
where for \(\forall t\), 
\begin{itemize}
  \item \(S_i(t) w^{S}_{ii} \beta_i (1 - u_i) \sum\limits_{j \ne i} w^{I}_{ji} I_j(t)\) characterizes the risk of infection for susceptible residents in region \(i\) resulting from contact with infected individuals who have traveled from other regions \(j \ne i\) to region \(i\) without adhering to the local testing policy enforced by region \(i\).

  \item \(S_i(t) \sum\limits_{j \ne i} w_{ij}^{S} \beta_j w^{I}_{jj} I_j(t)\) represents the exposure of susceptible individuals from region \(i\) who travel to other regions \(j \ne i\) and encounter infected residents in those destination regions.

  \item \(S_i(t) \sum\limits_{j \ne i} w_{ij}^{S} \beta_j \sum\limits_{k \ne j} (1 - u_j) w_{kj}^{I} I_k(t)\) accounts for the risk to susceptible individuals from region \(i\) who travel to regions \(j \ne i\), where they may be exposed to infected visitors from other regions \(k \ne j\) who are in region \(j\) without complying with that region’s testing protocol.

  \item \(S_i(t) w_{ii}^{S} \beta_i w_{ii}^{I} I_i(t)\) captures local disease transmission within region \(i\), arising from interactions between the susceptible and infected populations residing in the same region.

  \item \(\gamma_i I_i(t)\) represents the recovery of infected individuals in region \(i\), occurring at a constant rate \(\gamma_i\).
\end{itemize}
\begin{remark}
The term 
\(
S_i(t) \sum\limits_{j \neq i} w_{ij}^{S} \, \beta_j \sum\limits_{k \neq j} (1 - u_j) w_{kj}^{I} I_k(t)
\)
vanishes when we consider only two regions. More generally, this term is also negligible, since the number of infected visitors arriving from other regions is typically very small, especially under a strict testing policy in region \(j\).

\end{remark}
\subsection{Hybrid SEIR Model Without Vaccination}

We now extend the above SIR model to an SEIR framework by introducing a new compartment, \(E_i\), representing exposed  individuals in region \(i\). Unlike the SIR model, which assumes individuals become infectious immediately upon infection, the SEIR model incorporates an incubation period during which individuals are infected but not yet contagious. This modification requires us to reinterpretate certain parameters and introduce several new ones:
\begin{itemize}
    \item \(\beta_i\): the rate at which susceptible individuals become exposed through contact with  infectious individuals  in region \(i\);
    \item \(\sigma_i\): the rate at which exposed  individuals progress to become infected after an incubation (latent) period in region \(i\);
    \item \(w_{ji}^{E}\): the travel rate of exposed individuals moving from region \(j\) to region \(i\).
\end{itemize}

The evolution of the susceptible (\(S_i\)), exposed (\(E_i\)), infected (\(I_i\)), and recovered (\(R_i\)) populations in region \(i \in \{1, 2, \dots, N\}\) is governed by the following system of equations:
\begin{equation}
\label{eqn: SEIR}
\begin{aligned}    
    \frac{d}{dt} S_i(t) & =- S_i(t) w^{S}_{ii}(1-u_i)
    \beta_i \sum\limits_{j \not =i} w_{ji}^II_j(t)
     - S_i(t) \sum\limits_{j \neq i} w_{ij}^{S}
\beta_jw^{I}_{jj}I_j(t) \\
& \quad -S_i(t) \sum \limits_{j\not = i}w_{ij}^{S}\beta_j\sum\limits_{k \not = j} (1-u_j)w_{kj}^I I_k- S_i(t)w_{ii}^S \beta_i w_{ii}^I I_i(t),\\
\frac{d}{dt} E_i(t) & = S_i(t) w^{S}_{ii}(1-u_i)
    \beta_i \sum\limits_{j \not =i} w_{ji}^II_j(t)
     + S_i(t) \sum\limits_{j \neq i} w_{ij}^{S}
\beta_jw^{I}_{jj}I_j(t) \\
& \quad +S_i(t) \sum \limits_{j\not = i}w_{ij}^{S}\beta_j\sum\limits_{k \not = j} (1-u_j)w_{kj}^I I_k + S_i(t)w_{ii}^S \beta_i w_{ii}^I I_i(t) \\
& \quad \quad - \sigma_i E_i(t),\\
  \frac{d}{dt} I_i(t) & = \sigma_i E_i(t) -\gamma_i I_i(t),\\
 \frac{d}{dt} R_i(t) & = \gamma_i I_i(t),\\
\end{aligned}
\end{equation}
 where  the term \( \sigma_i E_i(t) \) captures the transition of exposed individuals in region \(i\) to the infected class at rate \(\sigma_i\), reflecting the end of the incubation period  at time \(t\). The remaining terms in 
  (\eqref{eqn: SEIR})
  are analogous to those in (\eqref{eqn: SIR}) with the primary distinction that the SEIR formulation explicitly incorporates an incubation phase during which individuals are infected but not yet infectious.
\subsection{Hybrid SEIR Model with Vaccination}

To model the transmission dynamics of infectious diseases in a partially vaccinated population distributed across multiple regions, we develop a hybrid SEIR (Susceptible--Exposed--Infectious--Recovered) framework that incorporates vaccination. This model extends the classical SEIR formulation by explicitly integrating vaccination status into the population structure. Specifically, individuals are stratified as either vaccinated or unvaccinated and are assumed to be mobile, traveling between regions either as residents or as short-term visitors. Each subgroup follows SEIR-type transitions, with transmission and recovery rates that vary based on both vaccination status and the geographic location in which individuals reside or travel.
\begin{figure}[ht]
    \centering
    \includegraphics[width=0.7\linewidth]{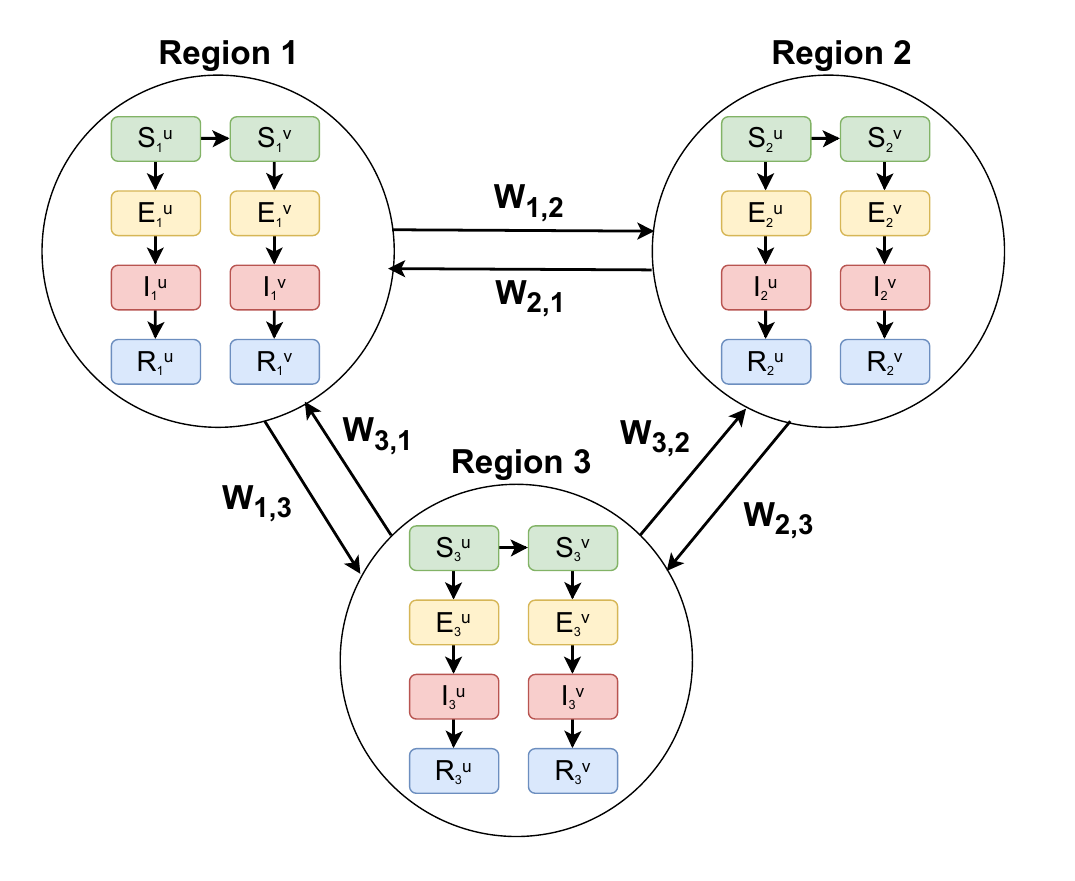}
    \caption[Hybrid SEIR Diagram]{A schematic representation of the hybrid SEIR model with separate compartments for vaccinated and unvaccinated individuals. This diagram demonstrates a scenario of three regions (denoted by circles) where each region contains populations of \(X_i^Y,\) where \(X \in \{S, E, I, R\}\), \(Y \in \{u, v\}\) and \(i \in \{1, 2, 3\}\). The arrows between each compartment within a region correspond to the flow of individuals through the Susceptible, Exposed, Infected, and Removed compartments. The temporary traveler rates \(W_{ab} = \sum \limits_{X \in \{S, E, I\}, Y \in \{u, v\}} w_{ab}^{(X,Y)} \) for \(a, b \in \{1, 2, 3\}\), representing the movement of susceptible, exposed and infectious individuals, are indicated by arrows between the circles.
}
    \label{fig:hybrid-SEIR-diagram}
\end{figure}
Building on the standard SEIR framework, we further stratify each compartment by vaccination status to account for heterogeneity in susceptibility, disease progression, and transmission potential. Specifically, for each disease stage \(X \in \{S, E, I, R\}\), we distinguish between vaccinated (\(X^v\)) and unvaccinated (\(X^u\)) individuals. This results in the following eight compartments:

\begin{table}[h!]
\centering
\begin{tabular}{ll}
\hline
\textbf{Compartment} & \textbf{Description} \\
\hline
\(S^u\) & Unvaccinated susceptible individuals who are fully at risk of infection \\
\(S^v\) & Vaccinated susceptible individuals  \\
 &  who are partially protected against infection \\
\(E^u\) & Unvaccinated exposed individuals who are infected but not yet infectious \\
\(E^v\) & Vaccinated exposed individuals in the incubation (latent) phase of infection \\
\(I^u\) & Unvaccinated infectious individuals capable of transmitting the disease \\
\(I^v\) & Vaccinated infectious individuals,  \\
 & typically with reduced transmissibility or shorter infectious periods \\
\(R^u\) & Unvaccinated individuals who have recovered or died \\
\(R^v\) & Vaccinated individuals who have recovered or died \\
\hline
\end{tabular}
\caption{Description of compartments in the SEIR model with vaccination stratification.}
\label{tab:vaccinated_seir}
\end{table}

To model interregional mobility, we define directional travel weights that are specific to both disease state and vaccination status. These weights capture the proportion of individuals in a given subpopulation who travel from one region to another per unit time. Specifically, the notation \(w_{ab}^{(X,Y)}\) denotes the travel rate from region \(a\) to region \(b\) for individuals in disease state \(X \in \{S, E, I\}\) and vaccination status \(Y \in \{u, v\}\), where \(u\) and \(v\) indicate unvaccinated and vaccinated, respectively.

The key travel weights used in the model are summarized in Table~\eqref{tab:travel_weights}.

\begin{table}[h!]
\centering
\begin{tabular}{lll}
\hline
\textbf{Notation} & \textbf{Description} & \textbf{Direction} \\
\hline
\(w_{ij}^{(S,u)}\) & Travel rate for unvaccinated susceptible individuals & From region \(i\) to region \(j\) \\
\(w_{ij}^{(S,v)}\) & Travel rate for vaccinated susceptible individuals   & From region \(i\) to region \(j\) \\
\(w_{ji}^{(E,u)}\) & Travel rate for unvaccinated exposed individuals & From region \(j\) to region \(i\) \\
\(w_{ji}^{(E,v)}\) & Travel rate for vaccinated exposed individuals   & From region \(j\) to region \(i\) \\
\(w_{ji}^{(I,u)}\) & Travel rate for unvaccinated infectious individuals & From region \(j\) to region \(i\) \\
\(w_{ji}^{(I,v)}\) & Travel rate for vaccinated infectious individuals   & From region \(j\) to region \(i\) \\
\hline
\end{tabular}
\caption{Directional travel weights by disease state and vaccination status.}
\label{tab:travel_weights}
\end{table}

The progression of individuals through epidemiological states is governed by a set of region- and status-specific parameters. These parameters control transitions between compartments based on contact rates, incubation periods, vaccination coverage, and recovery dynamics. In particular, we distinguish between vaccinated and unvaccinated individuals to account for differences in susceptibility, transmissibility, and disease progression.

Table~\eqref{tab:epi_parameters} summarizes the key epidemiological parameters used in the model, along with their interpretations.

\begin{table}[h!]
\centering
\begin{tabular}{ll}
\hline
\textbf{Parameter} & \textbf{Description} \\
\hline
\(\theta_i\) & Vaccination rate in region \(i\) \\
\(\beta_i^{uu}\) & Transmission rate from unvaccinated infectious  \\
 & to unvaccinated susceptible individuals in region \(i\) \\
\(\beta_i^{uv}\) & Transmission rate from vaccinated infectious \\
&  to unvaccinated susceptible individuals in region \(i\) \\
\(\beta_i^{vu}\) & Transmission rate from unvaccinated infectious \\
 & to vaccinated susceptible individuals in region \(i\) \\
\(\beta_i^{vv}\) & Transmission rate from vaccinated infectious \\
 &  to vaccinated susceptible individuals in region \(i\) \\
\(\sigma_i^u\) & Incubation rate for unvaccinated exposed individuals in region \(i\) \\
\(\sigma_i^v\) & Incubation rate for vaccinated exposed individuals in region \(i\) \\
\(\gamma_i^u\) & Recovery (or removal) rate for unvaccinated individuals in region \(i\) \\
\(\gamma_i^v\) & Recovery (or removal) rate for vaccinated individuals in region \(i\)\\
\hline
\end{tabular}
\caption{Epidemiological parameters defining transitions by disease and vaccination status.}
\label{tab:epi_parameters}
\end{table}

To incorporate vaccination into the spatial SEIR framework, we define the population in region \(i \in \{1, 2, \dots, N\}\) as \(X_i^Y(t)\), where \(X \in \{S, E, I, R\}\) denotes the disease stage—Susceptible, Exposed, Infectious, or Removed—and \(Y \in \{u, v\}\) indicates vaccination status, with \(u\) representing unvaccinated and \(v\) representing vaccinated individuals. The compartmental dynamics for each \(X_i^Y(t)\) are governed by the following system of differential equations: 
\begin{equation}
\begin{aligned} 
    \frac{d}{dt} S_i^u(t) 
    & = -\theta_i S_i^u(t)
    - S_i^u(t)w_{ii}^{S, u}\Big((1-u_i) \big(\beta_i^{uu} \sum \limits_{j \not = i} w_{ji}^{I,u}I_j^u(t) + \beta_i^{uv} \sum \limits_{j \not = i} w_{ji}^{I,v}I_j^v(t)\big)\Big) \\
    & \quad - S_i^u(t) \sum \limits_{j \not = i} w_{i,j}^{S,u} \big(\beta_{j}^{uu}w_{jj}^{I,u}I_j^u(t) + \beta_j^{uv}w_{jj}^{I,v}I_j^v(t)\big)\\
    & \quad -S_i^u(t) \Big(\sum \limits_{j \not =i} w_{ij}^{S,u} \Big(\beta_j^{uu} \sum\limits_{k \not =j}(1-u_j)w_{kj}^{I,u}I_k^u(t) + \beta_j^{uv}\sum\limits_{k \not =j}(1-u_j)w_{kj}^{I,v}I_k^v(t)\Big)\Big) \\
    & \quad - S_i^u(t)w_{ii}^{S,u}\big(\beta_i^{uu}w_{ii}^{I,u}I_i^u(t) + \beta_i^{uv}w_{ii}^{I,v}I_i^v(t)\big),\\
    \frac{d}{dt} S_i^v(t) & =\theta_i S_i^u(t)
    - S_i^v(t)w_{ii}^{S, v}\Big((1-u_i) \big(\beta_i^{vu} \sum \limits_{j \not = i} w_{ji}^{I,u}I_j^u(t) + \beta_i^{vv} \sum \limits_{j \not = i} w_{ji}^{I,v}I_j^v(t)\big)\Big) \\
    & \quad - S_i^v(t) \sum \limits_{j \not = i} w_{ij}^{S,v} \big(\beta_{j}^{vu}w_{jj}^{I,u}I_j^u(t) + \beta_j^{vv}w_{jj}^{I,v}I_j^v(t)\big)\\
    &  \quad -S_i^v(t) \Big(\sum \limits_{j \not =i} w_{i,j}^{S,v} \Big(\beta_j^{vu} \sum\limits_{k \not =j}(1-u_j)w_{kj}^{I,u}I_k^u(t) + \beta_j^{vv}\sum\limits_{k \not =j}(1-u_j)w_{kj}^{I,v}I_k^v(t)\Big)\Big) \\
    & \quad - S_i^v(t)w_{ii}^{S,v}\big(\beta_i^{vu}w_{ii}^{I,u}I_i^u(t) + \beta_i^{vv}w_{ii}^{I,v}I_i^v(t)\big),\\
     \frac{d}{dt} E_i^u(t) & = S_i^u(t)w_{ii}^{S, u}\Big((1-u_i) \big(\beta_i^{uu} \sum \limits_{j \not = i} w_{ji}^{I,u}I_j^u(t) + \beta_i^{uv} \sum \limits_{j \not = i} w_{ji}^{I,v}I_j^v(t)\big)\Big) \\
      & \quad + S_i^u(t) \sum \limits_{j \not = i} w_{i,j}^{S,u} \big(\beta_{j}^{uu}w_{jj}^{I,u}I_j^u(t) + \beta_j^{uv}w_{jj}^{I,v}I_j^v(t)\big)\\
    & \quad +S_i^u(t) \Big(\sum \limits_{j \not =i} w_{ij}^{S,u} \Big(\beta_j^{uu} \sum\limits_{k \not =j}(1-u_j)w_{kj}^{I,u}I_k^u(t) + \beta_j^{uv}\sum\limits_{k \not =j}(1-u_j)w_{kj}^{I,v}I_k^v(t)\Big)\Big) \\
     & \quad + S_i^u(t)w_{ii}^{S,u}\big(\beta_i^{uu}w_{ii}^{I,u}I_i^u(t) + \beta_i^{uv}w_{ii}^{I,v}I_i^v(t)\big) - \sigma_i^u E_i^u(t),\\
     \frac{d}{dt} E_i^v(t) & = S_i^v(t)w_{ii}^{S, v}\Big((1-u_i) \big(\beta_i^{vu} \sum \limits_{j \not = i} w_{ji}^{I,u}I_j^u(t) + \beta_i^{vv} \sum \limits_{j \not = i} w_{ji}^{I,v}I_j^v(t)\big)\Big) \\
      & \quad + S_i^v(t) \sum \limits_{j \not = i} w_{ij}^{S,v} \big(\beta_{j}^{vu}w_{jj}^{I,u}I_j^u(t) + \beta_j^{vv}w_{jj}^{I,v}I_j^v(t)\big)\\
    & \quad +S_i^v(t) \Big(\sum \limits_{j \not =i} w_{i,j}^{S,v} \Big(\beta_j^{vu} \sum\limits_{k \not =j}(1-u_j)w_{kj}^{I,u}I_k^u(t) + \beta_j^{vv}\sum\limits_{k \not =j}(1-u_j)w_{kj}^{I,v}I_k^v(t)\Big)\Big) \\
     & \quad + S_i^v(t)w_{ii}^{S,v}\big(\beta_i^{vu}w_{ii}^{I,u}I_i^u(t) + \beta_i^{vv}w_{ii}^{I,v}I_i^v(t)\big) - \sigma_i^v E_i^v(t),\\
    \frac{d}{dt} I_i^u(t) & =\sigma_i^u E_i^u(t) - \gamma_i^u I_i^u(t), \, 
     \frac{d}{dt} I_i^v(t)=\sigma_i^vE_i^v(t) - \gamma_i^v I_i^v(t), \\
     \frac{d}{dt} R_i^u(t)& =\gamma_i^u I_i^u(t), \, \frac{d}{dt} R_i^v(t)=\gamma_i^v I_i^v(t).
    \end{aligned}
    \label{eqn: SEIR_vacc}
\end{equation}

The dynamics in equation~\eqref{eqn: SEIR_vacc} extend those of the baseline model in equation~\eqref{eqn: SEIR} by explicitly stratifying each compartment by vaccination status and incorporating vaccination transitions, \(S_i^u \rightarrow S_i^v\), via the region-specific vaccination rate \(\theta_i\). A schematic illustration of the hybrid SEIR model incorporating vaccination within three regions is provided in Figure~\(\eqref{fig:hybrid-SEIR-diagram}\).

\section{Optimal Control Problems}\label{sec:optimal_control_problems}
While the full $N$-region SEIR framework with vaccination stratification (Section \ref{S2}) describes the real-world epidemiology, its high dimensionality prevents the analytical derivation of optimal control trajectories. To establish closed-form results, specifically the bang-bang structure of optimal controls via Pontryagin’s Maximum Principle, we restrict ourselves to the case of two regions. 
This assumption keeps a high level of generality of the analysis, since the second country can be regarded as a proxy for the rest of the world.
Moreover, the qualitative insights derived from the two-country model provide valuable information to implement and interpret numerical optimal control schemes on the expanded SEIR model.\\
Therefore, in this section, we investigate optimal testing strategies in a two-country setting. Our objective is to characterize the optimal testing policy for one country while assuming that the testing strategy of the other country is fixed.  We first formulate the optimal control problem over a general class of admissible (measurable) testing controls. We then restrict our focus to piecewise constant policies and characterize the corresponding optimal switching times.

\subsection{Analysis in the Measurable Control Setting}\label{sec:analysis_in_measurable_control_setting}
We study the dynamics governed by 
\eqref{eqn: SIR} 
with initial conditions 
\[
(S_i(0), I_i(0), R_i(0)) 
= (S_{i,0}, I_{i,0}, R_{i,0}) 
\in \mathbb{R}_{\geq 0}^3.
\]
Here, 
\[
u_i : [0,T] \to [0,1]
\]
denotes a measurable function representing the testing control in region 
\(i \in \{1, 2\}\), over a fixed, sufficiently large time horizon 
\(T > 0\). We assume that \(u_2\) is fixed and only \(u_1\) subject to optimization.
We aim to find a control $u_1$ that minimizes the cost functional:
\begin{align}\label{cost}
    \mathcal{J}(u_1)  = \int_0^T \big(u_1(t) c_{test} N_2(t) w_{21} + u_1(t) c_t  I_{2}(t) w_{21} \big)\,dt + c_d  d_{R} R_1(T)+  c_d d_{I} I_1(T)\,.
\end{align}
The running costs consist of two contributions:
\begin{itemize}
    \item \emph{Administrative testing cost.}  
    The parameter \(c_{\text{test}}\) represents the cost of administering a test to one incoming individual.  
    The term \(u_1(t)\, c_{\text{test}}\, N_2(t)\, w_{21}\) corresponds to the number of incoming individuals from region~2 (as encoded by the mobility weight \(w_{21}\)) who are tested at time \(t\).
    
    \item \emph{Cost of rejecting tourists.}  
    The parameter \(c_t\) denotes the economic loss associated with rejecting one infected tourist.  
    The term \(u_1(t)\, c_t\, I_2(t)\, w_{21}\) captures the expected number of rejected infected individuals from region~2 at time \(t\).
\end{itemize}

The terminal cost accounts for health outcomes in region~1: \(c_d\) is the cost associated with one death, \(d_R\) and \(d_I\) represent, respectively, the death rates within the Recovered and Infected compartments. Thus the terms \(c_d d_R R_1(T)\) and \(c_d d_I I_1(T)\) quantify the expected mortality cost at the end of the time horizon.\\

In the following, we use the notation 
\begin{align*}
    \lambda:[0&,T]\to\mathbb{R}^6,\\
    t&\mapsto \big(\lambda_1^1(t),\lambda_1^2(t),\lambda_1^3(t),\lambda_2^1(t),\lambda_2^2(t),\lambda_2^3(t)\big),
    \end{align*}
to indicate the adjoint variable for the optimization problem. In this setting the Hamiltonian is given by
\begin{align*}
    H(&{\bf x}(t), { u_1} (t),  {\bf\lambda} (t), t) = \\
     &\quad = \lambda_1^1\Big(- S_1(t) w^{S}_{11}
    \beta_1 {(1-u_1)}
    w^{I}_{21} I_{2}(t) - S_1(t)   w_{12}^{S} 
    \beta_2w^{I}_{22}I_2(t) -S_1(t)w_{11}^S \beta_1 w_{11}^I I_1(t)\label{1}\Big) \\
    & \quad + \lambda_1^2\Big(S_1(t) w^{S}_{11}
    \beta_1 {(1-u_1)}
    w^{I}_{21} I_{2}(t) + S_1(t)   w_{12}^{S} 
    \beta_2w^{I}_{22}I_2(t)+S_1(t)w_{11}^S \beta_1 w_{11}^I I_1(t)  -\gamma_1 I_1(t)\Big)\\
    & \quad + \lambda_1^3 \gamma_1 I_1(t)\\
    & \quad + \lambda_2^1\Big(- S_2(t) w^{S}_{22}
    \beta_2 {(1-u_2)}
    w^{I}_{12} I_{1}(t) - S_2(t)   w_{21}^{S} 
    \beta_1w^{I}_{11}I_1(t)-S_2(t)w_{22}^S \beta_2 w_{22}^I I_2(t)\Big) \\
    & \quad + \lambda_2^2\Big(S_2(t) w^{S}_{22}
    \beta_2 {(1-u_2)}
    w^{I}_{12} I_{1}(t) + S_2(t)   w_{21}^{S} 
    \beta_1w^{I}_{11}I_1(t)+S_2(t)w_{22}^S \beta_2 w_{22}^I I_2(t)-\gamma_2 I_2(t)\Big) \\
    & \quad + \lambda_2^3 \gamma_2 I_2(t) + u_1 c_{test} N_2 w_{21} + u_1 c_t  I_{2}(t) w_{21}.
    \end{align*}
    By applying the Pontryagin Maximum principle, we know that an optimal control satisfies
    $$u_1^*=\text{ argmin}_{u_1(t)\in[0,1]} H({\bf x}(t), { u_1} (t),  {\bf\lambda} (t), t).$$
Since the Hamiltonian is linear in the control variable \(u_1\), the minimizer can be determined by analyzing the sign of the switching function. In this context, we denote the switching function by $\phi(t)$, and it is given by
\begin{align*}
  \phi(t)=  (\lambda_1^1 - \lambda_1^2)S_1(t)w_{11}^S\beta_1w_{21}^II_2(t)
    + c_{test} N_{2,0} w_{21} +  c_t  I_{2}(t) w_{21}\,,
\end{align*}
where
\[N_{2,0}=S_{2,0}+ I_{2,0} + R_{2,0}=S_2(t)+I_2(t)+R_2(t)=N_2(t) .\]

The adjoint variables $\lambda_i$ satisfy the following system of differential equations:

\begin{align*}
    &\dot{\lambda}_1^1=(\lambda_1^1-\lambda_1^2)\Big(  w^{S}_{11}
    \beta_1 {(1-u_1)}
    w^{I}_{21}I_2(t)  +  w_{12}^{S} 
    \beta_2w^{I}_{22}I_2(t)+w_{11}^S \beta_1 w_{11}^I I_1(t)\Big) \\
    &\dot{\lambda}_1^2=(\lambda_1^1-\lambda_1^2)S_1(t)w_{11}^S \beta_1 w_{11}^I+(\lambda_1^2-\lambda_1^3 )\gamma_1 
   + (\lambda_2^1-\lambda_2^2)S_2(t)\Big( w^{S}_{21}
    \beta_1 
    w^{I}_{11} + {(1-u_2)} w_{22}^{S} 
    \beta_2w^{I}_{12}\Big)\\
     &\dot{\lambda}_1^3=0\\
     &\dot{\lambda}_2^1=(\lambda_2^1-\lambda_2^2)\Big(  w^{S}_{22}
    \beta_2 {(1-u_2)}
    w^{I}_{12}I_{1}(t)  +  w_{21}^{S} 
    \beta_1w^{I}_{11}I_{1}(t)+w_{22}^S \beta_2 w_{22}^I I_2(t)\Big) -u_1 c_{test}w_{21}\\
    &\dot{\lambda}_2^2=(\lambda_2^1-\lambda_2^2)S_2(t)w_{22}^S \beta_2 w_{22}^I+(\lambda_2^2-\lambda_2^3)\gamma_2 
   + (\lambda_1^1-\lambda_1^2)S_1(t)\Big( w^{S}_{11}
    \beta_1 {(1-u_1)}
    w^{I}_{21}+    w_{12}^{S} 
    \beta_2w^{I}_{22}\Big)\\ &\qquad \qquad-u_1 c_t   w_{21}-u_1 c_{test}w_{21}\\
        &\dot{\lambda}_2^3=-u_1 c_{test}w_{21}\\
    \end{align*}
 with final data
  $$\begin{cases}\lambda_1^1(T)=0,\\
    \lambda_1^2(T)=c_d d_I,\\
    \lambda_1^3(T)=c_d d_R,\\
    \lambda_2^1(T)=0,\\
    \lambda_2^2(T)=0,\\
    \lambda_2^3(T)=0.\\\end{cases} $$
It is immediate to check that 
$$\lambda_1^3(t)=c_d d_R, \qquad \lambda_2^3(t)=\int_t^Tu_1(s) c_{test}w_{21}ds. $$

However, the remaining adjoint variables are strongly interdependent, making explicit integration in closed form infeasible. Therefore, we analyze the sign of the switching function $\phi(t)$ at the final time $T$ by examining the order of magnitude of the variables involved. We rewrite $\phi(t)$ as follows:
\[
\phi(t) = A(t) + B(t), \qquad \text{ where } \quad \begin{cases}
A(t) = (\lambda_1^1 - \lambda_1^2)\, S_1(t)\, w_{11}^S\, \beta_1\, w_{21}^I\, I_2(t), \\
B(t) = c_{\text{test}}\, N_{2,0}\, w_{21} + c_t\, I_2(t)\, w_{21}.
\end{cases}
\]

We assume the following typical magnitudes for the variables:
\begin{itemize}
  \item $S_1(t), S_2(t) \sim 10^6$--$10^7$ (population sizes);
  \item $I_1(t), I_2(t) \sim 10^4$--$10^5$ (infected individuals);
  \item $\beta_1 \sim 10^0$ (usually between 1.5 and 2);
  \item $w_{11}^S \sim 10^0$;
  $w_{21}^I \sim 10^{-3}$;
  \item $c_{\text{test}}, c_t \sim 10^2$;
  {$c_{\text{death}} \sim 10^{6}$};
  {$d_r \sim 10^{-3}$};
  {$d_I \sim 10^{-3}$};
  \item $N_{2,0} \sim 10^7$--$10^8$;
  \item $w_{21} \sim 10^{-1}$--$10^0$ (depending on the travel intensity between the two countries).
\end{itemize}

With these estimates, $\phi(T)$ can be expressed as
\begin{align*}
  \phi(T) & =  (\lambda_1^1 - \lambda_1^2)\,S_1(T)\,w_{11}^S\,\beta_1\,w_{21}^I\,I_2(T)
    + c_{\text{test}}\,N_{2,0}\,w_{21} +  c_t\,I_{2}(T)\,w_{21},\\
    & \approx  - c_d d_I\,S_1(T)\,w_{11}^S\,\beta_1\,w_{21}^I\,I_2(T)
    + c_{\text{test}}\,N_{2,0}\,w_{21} +  c_t\,I_{2}(T)\,w_{21}.
\end{align*}

We observe that $A(T)$ remains dominant unless the number of infected individuals in country~2 is very low. This implies that it is optimal to continue testing until the pandemic is nearly resolved in country~2.

\medskip
\begin{remark}
We briefly comment on the possibility that the switching function $\phi(t)$ may vanish identically on a nontrivial time interval, i.e., $\phi(t) = 0$ for all $t \in [t_1, t_2]$ with $t_2 > t_1$. In this case, the control is said to be \emph{singular} over $[t_1, t_2]$. 

To determine whether such singular trajectories can occur, one must analyze higher-order derivatives of $\phi(t)$ with respect to time and verify whether the control appears explicitly in any of them. If the control does not appear in any finite-order derivative, the trajectory cannot be singular. 

In our setting, given the complex and nonlinear structure of $\phi(t)$ and the strong influence of state and adjoint variables, singular arcs are theoretically possible but unlikely to occur in practice unless the adjoint variables satisfy highly specific relations. Moreover, the positivity and dominance of the constant terms in $\phi(t)$ (notably $c_{\text{test}} N_{2,0} w_{21}$) make it improbable that $\phi(t)$ remains exactly zero over a nontrivial interval. Thus, while singular controls are a theoretical possibility, they are expected to be rarely observed in simulations of optimal trajectories.
\end{remark}

\subsection{Reduction to Piecewise Constant Controls}
\label{subsection: piecewise_controls}
The complexity of the switching function suggests that singular trajectories may rarely occur. For this reason we restrict the class of admissible control to the set of piecewise constant functions taking values in $\lbrace 0,1\rbrace$. Specifically, we will address the following three problems:
\begin{itemize}
    \item[i)] Given a scenario where no testing is implemented, what is the optimal time to start testing at the border?
    \item[ii)] In a scenario where testing is active, what is the optimal time to suspend it?
    \item[iii)] What are the two optimal times for activating and deactivating testing at the border?
\end{itemize}
\subsubsection{Case 1: Optimal Timing for Initiating Border Testing}
\label{Case_1}

In this case, we address Problem (i): determining the optimal time for country 1 to activate testing at the border. Specifically, we consider a piecewise-constant control $u_1(t)$ that switches from 0 to 1 at some time $t^*\in(0,T)$, representing the activation of the policy:

\begin{equation}\label{C1}
\tilde{u}_1(t) =
\begin{cases}
0, & \text{if } t \in [0, t^*),\\
1, & \text{if } t \in [t^*, T].
\end{cases}
\end{equation}
Our goal is to determine the optimal switching time $t^*$ that minimizes the cost functional \eqref{cost}.
We now present a characterization of the optimal switching time:

\begin{theorem}
The optimal switching time $t^*\in(0,T)$ for a control of the form \eqref{C1} satisfies the following implicit equation:
\begin{align}\label{OST1}
S_1(t^*) = \frac{ c_{test} N_2 w_{21} + I_2(t^*) c_t w_{21} }
{c_dd_I w_{11}^S \beta_1 w_{21}^I I_2(t^*) + c_dd_R\gamma_1 w_{11}^S \beta_1 w_{21}^I I_2(t^*)(T - t^*) }.
\end{align}
\end{theorem}

\begin{proof}
   Restricting to a control of the form \eqref{C1}, the functional becomes
$$ J(t^{*}) = \int_{t^*}^T( c_{test} N_2 w_{21}+  c_t  I_{2} w_{21})\, dt + c_d  d_{R} R_1(T) + c_d  d_{I} I_1(T) .$$
Differentiating \(J\) with respect to \(t^{*}\), we obtain 
\begin{align}\label{DJ}
     J'(t^{*}) =&  - c_{test} N_2 w_{21}-  c_t  I_{2}(t^{*}) w_{21} + c_d d_I \tfrac{dI_1(T)}{d t^{*}} + c_dd_R \tfrac{dR_1(T)}{d t^{*}}\notag\\
     = &  - c_{test} N_2 w_{21}-  c_t  I_{2}(t^{*}) w_{21} +c_dd_IS_1(t^*)w_{11}^S\beta_1w_{21}^I I_2(t^*)+ c_dd_R\gamma_1 S_1(t^{*}) w^{S}_{11}
    \beta_1 
    w^{I}_{21} I_{2}(t^{*})(T - t^{*})\notag\\
     = &  - c_{test} N_2 w_{21}-  c_t  I_{2}(t^{*}) w_{21} +c_dS_1(t^*)\Big(d_Iw_{11}^S\beta_1w_{21}^I I_2(t^*)+ d_R\gamma_1  w^{S}_{11}
    \beta_1 
    w^{I}_{21} I_{2}(t^{*})(T - t^{*})\Big).
\end{align}
Indeed even if apparently the terminal cost does not depend on $t^*$ we have
\begin{align*}I_1(T)&= I_1(t^*)+\int_{t^*}^T \frac{d I_1}{dt}\, dt \\
&= I_1(0)+\int_{0}^{t^*} \frac{d I_1}{dt} \, dt +\int_{t^*}^T \frac{d I_1}{dt}\, dt\\
& = I_1(0) + \int_0^{t^{*}}\big( S_1(t) w^{S}_{11}
    \beta_1 
    w^{I}_{21} I_{2}(t) + S_1(t)w_{12}^{S} 
\beta_2w^{I}_{22}I_2(t)+S_1(t)w_{11}^S \beta_1 w_{11}^I I_1(t)  -\gamma_1 I_1(t) \big) \, dt \\
& \quad + \int_{t^{*}}^T \big(S_1(t)w_{12}^{S} 
\beta_2w^{I}_{22}I_2(t)+S_1(t)w_{11}^S \beta_1 w_{11}^I I_1(t)  -\gamma_1 I_1(t) \big)dt 
\end{align*}

from which
\begin{align}\label{DI1}
    \tfrac{dI_1(T)}{d t^{*}} = S_1(t^{*}) w^{S}_{11}
    \beta_1 
    w^{I}_{21} I_{2}(t^{*}).
\end{align}

Analogously for $R_1(T)$ we get
\begin{align*}
R_1(T)&=R_1(t^*)+\int_{t^*}^T \gamma_1 I_1(t)\, dt\\
        &=R_1(t^*)+\int_{t^*}^T \gamma_1 \Big( I_1(t^*)+\int_{t^*}^t \frac{d I_1}{ds}\, ds\Big)\,dt\\
        &=R_1(t^*)+\int_{t^*}^T \gamma_1 \Big( I_1(0)+\int_{0}^{t^*} \frac{d I_1}{ds}\, ds+\int_{t^*}^t \frac{d I_1}{ds}\, ds\Big)\,dt\\
        &=R_1(0)+\int_{0}^{t^*} \gamma_1 I_1(t)\,dt\,+\int_{t^*}^T \gamma_1 \Big( I_1(0)+\int_{0}^{t^*} \frac{d I_1}{ds}\, ds+\int_{t^*}^t \frac{d I_1}{ds}\, ds\Big)\,dt.
\end{align*}

hence, differentiating with respect to \(t^{*}\), 
\begin{align}
    \tfrac{dR_1(T)}{dt^{*}} & =  \gamma_1 I_1(t^{*}) -\gamma_1 \Big( I_1(0)+\int_{0}^{t^*} \frac{d I_1}{ds}\, ds+\int_{t^*}^{t^{*}} \frac{d I_1}{ds}\, ds\Big) + \int_{t^*}^T \gamma_1 \Big(\frac{d I_1(t^{*}-)}{ds} - \frac{d I_1(t^{*}+)}{ds}\Big)\,dt \notag\\
    & =  \gamma_1 I_1(t^{*}) -\gamma_1 \Big( I_1(0)+\int_{0}^{t^*} \frac{d I_1}{ds}\, ds\Big) + \gamma_1 S_1(t^{*}) w^{S}_{11}
    \beta_1 
    w^{I}_{21} I_{2}(t^{*})(T - t^{*})\notag\\
     & =  \gamma_1 I_1(t^{*}) -\gamma_1 \Big( I_1(0)+I_1(t^{*}) - I_1(0)\Big) + \gamma_1 S_1(t^{*}) w^{S}_{11}
    \beta_1 
    w^{I}_{21} I_{2}(t^{*})(T - t^{*})\notag\\
     & = \gamma_1 S_1(t^{*}) w^{S}_{11}
    \beta_1 
    w^{I}_{21} I_{2}(t^{*})(T - t^{*}).\label{DR1}
\end{align}

By setting the derivative \eqref{DJ} equal to zero, we obtain that a critical time $t^*$ satisfies
\begin{align}
    S_1(t^*) & = \frac{ c_{test} N_2 w_{21} + I_2(t^*) c_t w_{21}}
{c_dd_I w_{11}^S \beta_1 w_{21}^I I_2(t^*) + c_dd_R\gamma_1 w_{11}^S \beta_1 w_{21}^I I_2(t^*)(T - t^*) }.
\label{critical_time_1}
\end{align} 
\end{proof}

\begin{remark}\label{remark case 1} Reading \eqref{critical_time_1} more carefully, we can highlight several interpretative aspects:

\begin{itemize}
    \item {\bf Interpretation of the critical time \(t^{*}\).}
    The numerator of \eqref{critical_time_1} represents the total cost associated with testing and tourism rejection, consisting of two components. The first term, \( w_{21} c_{test} N_2 \), represents the cost that region 1 for testing incoming individuals from region 2, while the second term, \(w_{21} c_t I_2(t^*)\), shows the cost to region 1 due to the rejection of infected tourists from region 2.
Similarly, the denominator of \eqref{critical_time_1} represents the total cost associated with deaths, including contributions from both infected and recovered individuals. The term \(c_dd_I w_{11}^S \beta_1 w_{21}^I I_2(t^*)\) captures the cost in region 1 induced by infections imported from region 2 at time \(t^{*}\), whereas, \(c_dd_R\gamma_1 w_{11}^S \beta_1 w_{21}^I I_2(t^*)(T - t^*)\) is the recovery-phase cost to region 1 over the remaining time horizon generated by infection pressure from region 2 at time \(t^{*}\). 
Consequently, at the critical time \( t^{*} \), the susceptible population in region 1 is determined by the ratio between the total cost associated with testing and tourism rejection and the total cost related to death. A higher cost due to testing and tourism rejection, together with a lower cost associated with death, will result in a larger susceptible population in region 1 at the critical time \( t^{*} \).

\item {\bf Relation between \(S_1(t^{*})\) and \(I_2(t^{*})\).}
Notice that \eqref{critical_time_1} can be rewritten as
\begin{align*}
    S_1(t^{*}) 
    = \frac{ c_{\text{test}} w_{21} \tfrac{N_2}{I_2(t^{*})} + c_t w_{21} }
    { c_{dd_I} w_{11}^S \beta_1 w_{21}^I + c_{dd_R} \gamma_1 w_{11}^S \beta_1 w_{21}^I (T - t^{*}) },
\end{align*}
which indicates that an increase in the infected population in region 2 leads to a decrease in the susceptible population in region 1 at the critical time \( t^{*} \). Intuitively, a larger \(I_2(t^{*})\) means a stronger external infection pressure and higher risk of imported infections, leading to a smaller remaining susceptible portion in region 1 at the activation threshold. Conversely, if \(I_2(t^{*})\) is small, the fraction \(\tfrac{N_2}{I_2(t^{*})}\) is large, so \(S_1(t^{*})\) is higher, which means that region 1 holds a larger susceptible population at the time \(t^{*}\). 

\item {\bf Relation between \(S_1(t^{*})\) and \(T\).}
A longer time horizon \(T\), meaning that the testing control window extends further into the future, results in a smaller susceptible population \(S_1(t^{*})\) at the critical time. In other words, when planning over a longer period, region 1 must reach the critical activation time at a lower susceptible level, since the potential recovery-related costs accumulated over the extended horizon are larger.

\item {\bf Existence of the critical activation time.}
Let 
\[
F(t) =  
\frac{ c_{\text{test}} N_2 w_{21} + I_2(t)\, c_t w_{21} }
{ c_{dd_I} w_{11}^S \beta_1 w_{21}^I I_2(t) 
+ c_{dd_R} \gamma_1 w_{11}^S \beta_1 w_{21}^I I_2(t) (T - t) }, t \in [0, T]. 
\]
The function \( F \colon [0, T] \to \mathbb{R}_{>0} \) is continuous, and the susceptible population in region 1, \( S_1(t) \), is also continuous and monotonically decreasing in time.  
If \( S_1(0) > F(0) \) and \( S_1(T) < F(T) \), then there exists some \( t^{*} \in (0, T) \) such that
\[
S_1(t^{*}) = F(t^{*}).
\]
\end{itemize}

\end{remark}

\subsubsection{Case 2: Optimal Timing for Suspending Border Testing}
We now address problem ii), considering a piecewise constant control $u_1$ with a switching time \(t^{*} \in (0, T)\) of the form 
\begin{equation}\label{C2}
    \bar{u}_1(t) =\begin{cases} 
    1 \text{ if } t \in [0, t^{*})\\
    0\text{ if } t \in [t^{*}, T]
    \end{cases}
    \end{equation}
Our goal is to determine the optimal switching time \(t^*\in (0, T)\) that minimizes \eqref{cost}.
\begin{theorem}
The optimal switching time $t^*\in(0,T)$ for a control of the form \eqref{C2} satisfies the following implicit equation:
\begin{align}\label{OST1}
S_1(t^*) = \frac{ c_{test} N_2 w_{21} + I_2(t^*) c_t w_{21} }
{c_dd_I w_{11}^S \beta_1 w_{21}^I I_2(t^*) + c_dd_R\gamma_1 w_{11}^S \beta_1 w_{21}^I I_2(t^*)(T - t^*) }.
\end{align}
\end{theorem}
\begin{proof}
Restricting to a control of the form \eqref{C2}, the functional becomes:
$$ J(t^{*}) = \int_0^{t^*}(c_{test} N_2 w_{21}+ c_t  I_{2} w_{21})\, dt + c_d  d_{R} R_1(T) + c_d  d_{I} I_1(T) .$$
Differentiating \(J\) with respect to \(t^{*}\), we obtain 
\begin{align}
     J'(t^{*}) =&  c_{test} N_2 w_{21}+ c_t  I_{2}(t^{*}) w_{21} + c_d d_I \tfrac{d I_1(T)}{d t^{*}} + c_d d_R \tfrac{d R_1(T)}{d t^{*}} \notag\\
     = &   c_{test} N_2 w_{21}+ c_t  I_{2}(t^{*}) w_{21} -c_d d_IS_1(t^*)w_{11}^S\beta_1w_{21}^I I_2(t^*)- c_d d_R\gamma_1 S_1(t^{*}) w^{S}_{11}
    \beta_1 
    w^{I}_{21} I_{2}(t^{*})(T - t^{*}).\label{DJ2}
\end{align}
Indeed, both $I_1(T)$ and $R_1(T)$ in the terminal cost depend on $t^*$. Looking first at  $I_1(T)$ we have
\begin{align*}I_1(T)&= I_1(t^*)+\int_{t^*}^T \frac{d I_1}{dt}\, dt \\
&= I_1(0)+\int_{0}^{t^*} \frac{d I_1}{dt} \, dt+\int_{t^*}^T \frac{d I_1}{dt}\, dt\\
& = I_1(0) + \int_0^{t^{*}}\Big( S_1(t)   w_{12}^{S} 
    \beta_2w^{I}_{22}I_2(t) { + S_1(t)w_{11}^S \beta_1 w_{11}^I I_1(t)} -\gamma_1 I_1(t)\Big) dt \\
& \quad + \int_{t^{*}}^T \big(S_1(t) w^{S}_{11}
    \beta_1 
    w^{I}_{21} I_{2}(t)+S_1(t)   w_{12}^{S} 
    \beta_2w^{I}_{22}I_2(t) { + S_1(t)w_{11}^S \beta_1 w_{11}^I I_1(t)} -\gamma_1 I_1(t)\big)dt 
\end{align*}
from which
\begin{align*}
    \tfrac{dI_1(T)}{d t^{*}} = - S_1(t^{*}) w^{S}_{11}
    \beta_1 
    w^{I}_{21} I_{2}(t^{*}).
\end{align*}
For $R_1(T)$ we have
\begin{align*}
R_1(T)&=R_1(t^*)+\int_{t^*}^T \gamma_1 I_1(t)\, dt\\
        &=R_1(t^*)+\int_{t^*}^T \gamma_1 \Big( I_1(t^*)+\int_{t^*}^t \frac{d I_1}{ds}\, ds\Big)\,dt\\
        &=R_1(t^*)+\int_{t^*}^T \gamma_1 \Big( I_1(0)+\int_{0}^{t^*} \frac{d I_1}{ds}\, ds+\int_{t^*}^t \frac{d I_1}{ds}\, ds\Big)\,dt\\
        &=R_1(0)+\int_{0}^{t^*} \gamma_1 I_1(t)\,dt\,+\int_{t^*}^T \gamma_1 \Big( I_1(0)+\int_{0}^{t^*} \frac{d I_1}{ds}\, ds+\int_{t^*}^t \frac{d I_1}{ds}\, ds\Big)\,dt,
\end{align*}
whose derivative with respect to \(t^{*}\) is
\begin{align*}
    \tfrac{dR_1(T)}{dt^{*}} & =  \gamma_1 I_1(t^{*}) -\gamma_1 \Big( I_1(0)+\int_{0}^{t^*} \frac{d I_1}{ds}\, ds+\int_{t^*}^{t^{*}} \frac{d I_1}{ds}\, ds\Big) + \int_{t^*}^T \gamma_1 \Big(\frac{d I_1(t^{*}-)}{ds} - \frac{d I_1(t^{*}+)}{ds}\Big)\,dt \\
    & =  \gamma_1 I_1(t^{*}) -\gamma_1 \Big( I_1(0)+\int_{0}^{t^*} \frac{d I_1}{ds}\, ds\Big) - \gamma_1 S_1(t^{*}) w^{S}_{11}
    \beta_1 
    w^{I}_{21} I_{2}(t^{*})(T - t^{*})\\
    & =- \gamma_1 S_1(t^{*}) w^{S}_{11}
    \beta_1 
    w^{I}_{21} I_{2}(t^{*})(T - t^{*}).
    \end{align*}
By setting the derivative \eqref{DJ2} equal to zero, we obtain that a critical time $t^*$ satisfies

\begin{align}
   S_1(t^*) = \frac{ c_{test} N_2 w_{21} + I_2(t^*) c_t w_{21} }
{c_dd_I w_{11}^S \beta_1 w_{21}^I I_2(t^*) + c_dd_R\gamma_1 w_{11}^S \beta_1 w_{21}^I I_2(t^*)(T - t^*) }.
\label{critical_time_2}
\end{align}

\end{proof}
\begin{remark}
    The equilibrium equation ~\eqref{critical_time_2} determining the optimal time for suspending border testing has the same form as the equilibrium equation ~\eqref{critical_time_1} for the critical activation time in Case~1. In other words, the critical switching time can be determined by examining the relationship between the susceptible population in region~1, \(S_1\), and the infected population in region~2, \(I_2\), as described in Remark~\ref{remark case 1}.
\end{remark}

\subsubsection{Case 3: Joint Optimization of Activation and Deactivation of Border Testing}
We now consider a control \( \hat{u}_1 \) that is piecewise constant, switching from 0 to 1 at time \( t_1 \in (0, T) \), and back to 0 at time \( t_2 \in (t_1, T) \). In other words, the travel policy is implemented during the time interval \( [t_1, t_2] \), and inactive otherwise. That is,

\begin{equation}\label{C3}
    \hat{u}_1(t) =\begin{cases} 
    0 \text{ if } t \in [0, t_1)\\
    1 \text{ if } t \in [t_1, t_2]\\
    0 \text{ if } t \in [t_2, T].
    \end{cases}
    \end{equation}
We would like to find the optimal switching time \(t_1^{*}, t_2^{*}\) over \((0, T)\) that minimize \eqref{cost}.
\begin{theorem}
The optimal switching time $t_1^*,t_2^*\in(0,T)$ for a control of the form \eqref{C3} satisfies the following implicit equation:
\begin{align*}
    S_1(t) & = \frac{ c_{test} N_2 w_{21}+  c_t  I_{2}(t) w_{21}}{ c_d d_I w^{S}_{11}
    \beta_1 
    w^{I}_{21} I_{2}(t)
     + c_d d_R \gamma_1 w^{S}_{11}
    \beta_1
    w^{I}_{21} I_{2}(t) (T - t)}\,.
\end{align*}
\end{theorem}
\begin{proof}
In the class of admissible controls of the form \eqref{C3}, the cost functional becomes:
$$ J(t_1, t_2) = \int_{t_1}^{t_2}( c_{test} N_2 w_{21}+  c_t  I_{2}(t)w_{21}))\, dt + c_d  d_{R} R_1(T) + c_d  d_{I} I_1(T) .$$

Note that both $I_1(T)$ and $R_1(T)$ are implicitely dependent on $t_1$ and $t_2$. Indeed for $I_1(T)$ we have
\begin{align*} 
    I_1(T) = & I_1(0) + \int_0^{t_1} \frac{d I_1}{dt}\, dt + \int_{t_1}^{t_2} \frac{d I_1}{dt}\, dt + \int_{t_2}^{T} \frac{d I_1}{dt}\, dt\\
    =& I_1(0) + \int_0^{t_1} \Big(S_1(t) w^{S}_{11}
    \beta_1 
    w^{I}_{21} I_{2}(t) + S_1(t)   w_{12}^{S} 
    \beta_2w^{I}_{22}I_2(t) { + S_1(t)w_{11}^S \beta_1 w_{11}^I I_1(t)} -\gamma_1 I_1(t)\Big)\, dt \\
    & \quad + \int_{t_1}^{t_2} \Big( S_1(t)   w_{12}^{S} 
    \beta_2w^{I}_{22}I_2(t) { + S_1(t)w_{11}^S \beta_1 w_{11}^I I_1(t)} -\gamma_1 I_1(t)\Big)\, dt\\
     & \quad + \int_{t_2}^{T}  \Big(S_1(t) w^{S}_{11}
    \beta_1 
    w^{I}_{21} I_{2}(t) + S_1(t)   w_{12}^{S} 
    \beta_2w^{I}_{22}I_2(t) { + S_1(t)w_{11}^S \beta_1 w_{11}^I I_1(t)} -\gamma_1 I_1(t)\Big) \, dt\\
\end{align*}

which implies 

\begin{align*}
    \tfrac{\partial I_1(T)}{\partial t_1} 
    = S_1(t_1) w^{S}_{11}
    \beta_1 
    w^{I}_{21} I_{2}(t_1),
\end{align*}
\begin{align*}
    \tfrac{\partial I_1(T)}{\partial t_2} 
    = -  S_1(t_2) w^{S}_{11}
    \beta_1 
    w^{I}_{21} I_{2}(t_2).
\end{align*}

On the other hand for $R_1(T)$ we have

\begin{align*}
    R_1(T) = & R_1(0) + \int_0^{t_1} \frac{d R_1}{dt}\, dt + \int_{t_1}^{t_2} \frac{d R_1}{dt} \, dt + \int_{t_2}^{T} \frac{d R_1}{dt}\, dt\\
     =& R_1(0) + \int_0^{t_1}\gamma_1 I_1(t)\, dt + \int_{t_1}^{t_2} \gamma_1 I_1(t)\, dt + \int_{t_2}^{T} \gamma_1 I_1(t)\, dt\\
      =& R_1(0) + \gamma_1 \int_0^{t_1} I_1(t)\, dt + \gamma_1 \int_{t_1}^{t_2} \Big(I_1(0) + \int_{0}^{t_1} \tfrac{d I_1(s)}{ds} \, ds + \int_{t_1}^{t} \tfrac{d I_1(s)}{ds} \, ds\Big)\, dt\\
      & \quad + \gamma_1\int_{t_2}^{T}  \Big(I_1(0) + \int_{0}^{t_1} \tfrac{d I_1(s)}{ds} \, ds + \int_{t_1}^{t_2} \tfrac{d I_1(s)}{ds} \, ds +\int_{t_2}^{t} \tfrac{d I_1(s)}{ds} \Big) \, dt
\end{align*}

which implies 

\begin{align*}
    \tfrac{\partial R_1(T)}{\partial t_1} & =\gamma_1 I_1(t_1) - \gamma_1 \Big(I_1(0) + \int_{0}^{t_1} \tfrac{d I_1(s)}{ds} \, ds \Big) + \gamma_1 \int_{t_1}^{t_2} \Big( \tfrac{d I_1(t_{1-})}{ds} -  \tfrac{d I_1(t_{1+})}{ds} \Big)\, dt  \\
    & \quad + \gamma_1 \int_{t_2}^T \Big(\tfrac{dI_1(t_{1-})}{ds} - \tfrac{dI_1(t_{1+})}{ds}\Big) \, dt \\
    & =\gamma_1 \int_{t_1}^{t_2} \Big( \tfrac{d I_1(t_{1-})}{ds} -  \tfrac{d I_1(t_{1+})}{ds} \Big)\, dt + \gamma_1 \int_{t_2}^T \Big(\tfrac{dI_1(t_{1-})}{ds} - \tfrac{dI_1(t_{1+})}{ds}\Big) \, dt 
\\
    \tfrac{\partial R_1(T)}{\partial t_2} & =\gamma_1 \Big(I_1(0) + \int_{0}^{t_1} \tfrac{d I_1(s)}{ds} \, ds + \int_{t_1}^{t_2} \tfrac{d I_1(s)}{ds} \, ds\Big) - \gamma_1 \Big(I_1(0) + \int_{0}^{t_1} \tfrac{d I_1(s)}{ds} \, ds + \int_{t_1}^{t_2} \tfrac{d I_1(s)}{ds} \, ds \Big)\\
    & \quad + \gamma_1\int_{t_2}^T \Big( \tfrac{d I_1(t_{2-})}{ds} - \tfrac{d_1{t_{2+}}}{ds}\Big)\, dt\\
    & =  \gamma_1\int_{t_2}^T \Big( \tfrac{d I_1(t_{2-})}{ds} - \tfrac{dI_1{t_{2+}}}{ds}\Big)\, dt
\end{align*}

Noticing that 

\begin{align*}
    \frac{d}{dt} I_1(t_{1-}) & =S_1(t_1) w^{S}_{11}
    \beta_1
    w^{I}_{21} I_{2}(t_1) + S_1(t_1)   w_{12}^{S} 
    \beta_2w^{I}_{22}I_2(t_1)+ S_1(t)w_{11}^S \beta_1 w_{11}^I I_1(t)  -\gamma_1 I_1(t_1)\\
    \frac{d}{dt} I_1(t_{1+}) & =S_1(t_1)   w_{12}^{S} 
    \beta_2w^{I}_{22}I_2(t_1) +S_1(t)w_{11}^S \beta_1 w_{11}^I I_1(t)+  -\gamma_1 I_1(t_1)\\
    \frac{d}{dt} I_1(t_{2-}) & =S_1(t_2)   w_{12}^{S} 
    \beta_2w^{I}_{22}I_2(t_2)+S_1(t)w_{11}^S \beta_1 w_{11}^I I_1(t)  -\gamma_1 I_1(t_2)\\
    \frac{d}{dt} I_1(t_{2+}) & =S_1(t_2) w^{S}_{11}
    \beta_1 
    w^{I}_{21} I_{2}(t_2) + S_1(t_2)   w_{12}^{S} 
    \beta_2w^{I}_{22}I_2(t_2)+S_1(t)w_{11}^S \beta_1 w_{11}^I I_1(t)  -\gamma_1 I_1(t_2).
\end{align*}
Therefore, 

\begin{align*}
    \tfrac{d I_1(t_{1-})}{ds} -  \tfrac{d I_1(t_{1+})}{ds} & =S_1(t_1) w^{S}_{11}
    \beta_1
    w^{I}_{21} I_{2}(t_1)\,, \\
    \tfrac{d I_1(t_{2-})}{ds} -  \tfrac{d I_1(t_{2+})}{ds}  &= -S_1(t_2) w^{S}_{11}
    \beta_1 
    w^{I}_{21} I_{2}(t_2)\,, 
\end{align*}
and thus 
\begin{align*}
    \tfrac{\partial R_1(T)}{\partial t_1} & = \gamma_1 S_1(t_1) w^{S}_{11}
    \beta_1
    w^{I}_{21} I_{2}(t_1) (T - t_1)\,,\\
   \tfrac{\partial R_1(T)}{\partial t_2} & = -\gamma_1S_1(t_2) w^{S}_{11}
    \beta_1 
    w^{I}_{21} I_{2}(t_2) (T-t_2)\,.\\ 
\end{align*}
Taking the partial derivative of \(J\) with respect to \(t_1\) and \(t_2\), and setting them equal to zero we get

\begin{align*}
     \tfrac{\partial J(t_1)}{\partial t_1} & = - c_{test} N_2 w_{21}-  c_t  I_{2}(t_1) w_{21} + c_d d_I I_1(T)'(t_1) + c_dd_R R_1(T)'(t_1)\\
     & = -c_{test} N_2 w_{21}-  c_t  I_{2}(t_1) w_{21}\\
     &\quad  + c_d d_IS_1(t_1) w^{S}_{11}
    \beta_1 
    w^{I}_{21} I_{2}(t_1) \\
     & \quad + c_d d_R \gamma_1 S_1(t_1) w^{S}_{11}
    \beta_1
    w^{I}_{21} I_{2}(t_1) (T - t_1)=0
\end{align*}

\begin{align*}
     \tfrac{\partial J(t_2)}{\partial t_2}  & =   c_{test} N_2 w_{21}+  c_t  I_{2}(t_2) w_{21} + c_d d_I I_1(T)'(t_2) + c_dd_R R_1(T)'(t_2)\\
     & = c_{test} N_2 w_{21}+  c_t  I_{2}(t_2) w_{21}\,\\
     & \quad -c_d d_I S_1(t_2) w^{S}_{11}
    \beta_1 
    w^{I}_{21} I_{2}(t_2)\\
     & \quad -c_d d_R \gamma_1S_1(t_2) w^{S}_{11}
    \beta_1 
    w^{I}_{21} I_{2}(t_2) (T-t_2)=0,
\end{align*}

from which 

\begin{align*}
    S_1(t_1) & = \frac{ c_{test} N_2 w_{21}+  c_t  I_{2}(t_1) w_{21}}{ c_d d_I w^{S}_{11}
    \beta_1 
    w^{I}_{21} I_{2}(t_1)
     + c_d d_R \gamma_1 w^{S}_{11}
    \beta_1
    w^{I}_{21} I_{2}(t_1) (T - t_1)}\,,\\
    S_1(t_2) & = \frac{ c_{test} N_2 w_{21}+ c_t  I_{2}(t_2) w_{21}}{ \quad c_d d_I  w^{S}_{11}
    \beta_1 
    w^{I}_{21} I_{2}(t_2)
     +c_d d_R \gamma_1 w^{S}_{11}
    \beta_1 
    w^{I}_{21} I_{2}(t_2) (T-t_2)}.
\end{align*}

\end{proof}

\begin{remark}
Both switching times \(t_1\) and \(t_2\) satisfy equilibrium conditions of the same algebraic form.  
Their roles, however, are fundamentally different: \(t_1\) corresponds to a switch from zero control to full control, whereas \(t_2\) corresponds to a switch from full control back to zero.  
Thus, although the equations defining them are structurally identical, the underlying dynamics in which they arise are not.
\end{remark}

\section{Numerical simulations}\label{S4}

\subsection{The model}

To analyze the properties of the optimal policies under the piecewise-constant control framework introduced in Section~\ref{subsection: piecewise_controls}, the SIR model without vaccination, given by \eqref{eqn: SIR}, was implemented in a two-country setting. The optimization is carried out with respect to the cost functional defined in \eqref{cost}, which quantifies the effectiveness of testing strategies across the two countries. The parameter values used throughout this section are provided in Section~\ref{sec:analysis_in_measurable_control_setting}. We begin by summarizing the numerical implementation of the model, then describe the minimization of the cost functional for the three classes of admissible control sets introduced in Section~\ref{subsection: piecewise_controls}, and finally discuss the behavior of the cost as a function of the switching time in each control scenario.

\subsection{Model Implementation}
The system of ODEs given by \eqref{eqn: SIR} is simulated using standard approaches. The model is implemented in Python using scientific computing libraries, SciPy and NumPy \cite{2020SciPy-NMeth, harris2020array}, and the system of ODEs is solved using SciPy's \textbf{solve\_ivp} function. To organize the model setup, a Python class is created that stores the model parameters, initial conditions, and cost outputs from the simulations.
To simplify the implementation of the numerical optimization, the piecewise-constant testing controls (\(u_i\)) are implemented using NumPy’s Heaviside function, with appropriate left or right shifts to represent the switching times. The cost functional is then interpreted as a function of the switching time of the testing policy, and the resulting scalar function is minimized numerically. For comparison with the analytical switching times, the zeros of the implicit equation~\eqref{OST1} are computed using SciPy’s fsolve routine. \\
In the Case~1 scenario, the system of ODEs \eqref{eqn: SIR} is solved without any control, and the populations \(S_1, I_1, I_2\) are treated as time-dependent functions. This is consistent with the definition of Case~1, where no intervention occurs for \(t \in [0, t^{*}]\), so the dynamics coincide with those of the uncontrolled system over this interval.
Similarly, in the Case~2 scenario, the system is solved with the control active from the beginning, and the corresponding populations \(S_1, I_1, I_2\) are interpreted as functions of time in the implicit equation. 
An analogous approach is adopted for the Case~3 scenario. First, the critical time \(t_1^{*}\), corresponding to the onset of control, is obtained in the same way as in Case~1. The system is then solved with the control activated at \(t_1^{*}\) and maintained for the remainder of the time horizon. The resulting populations \(S_1, I_1, I_2\) are subsequently treated as functions of time to compute the second critical time \(t_2^{*}\) using equation~\eqref{OST1}.\\
The results of the model simulation are presented as in Figure \ref{fig:simulation_output}, where the left axis gives information for the compartmental populations for Regions 1 and 2, and the right axis provides the value of the control throughout the simulation. Additionally, the solution to \eqref{OST1} is represented by a vertical dotted black line.

\begin{figure}[!htb]
    \centering
    \includegraphics[width=0.9\linewidth, trim={6cm, 2cm, 5cm, 2cm}, clip]{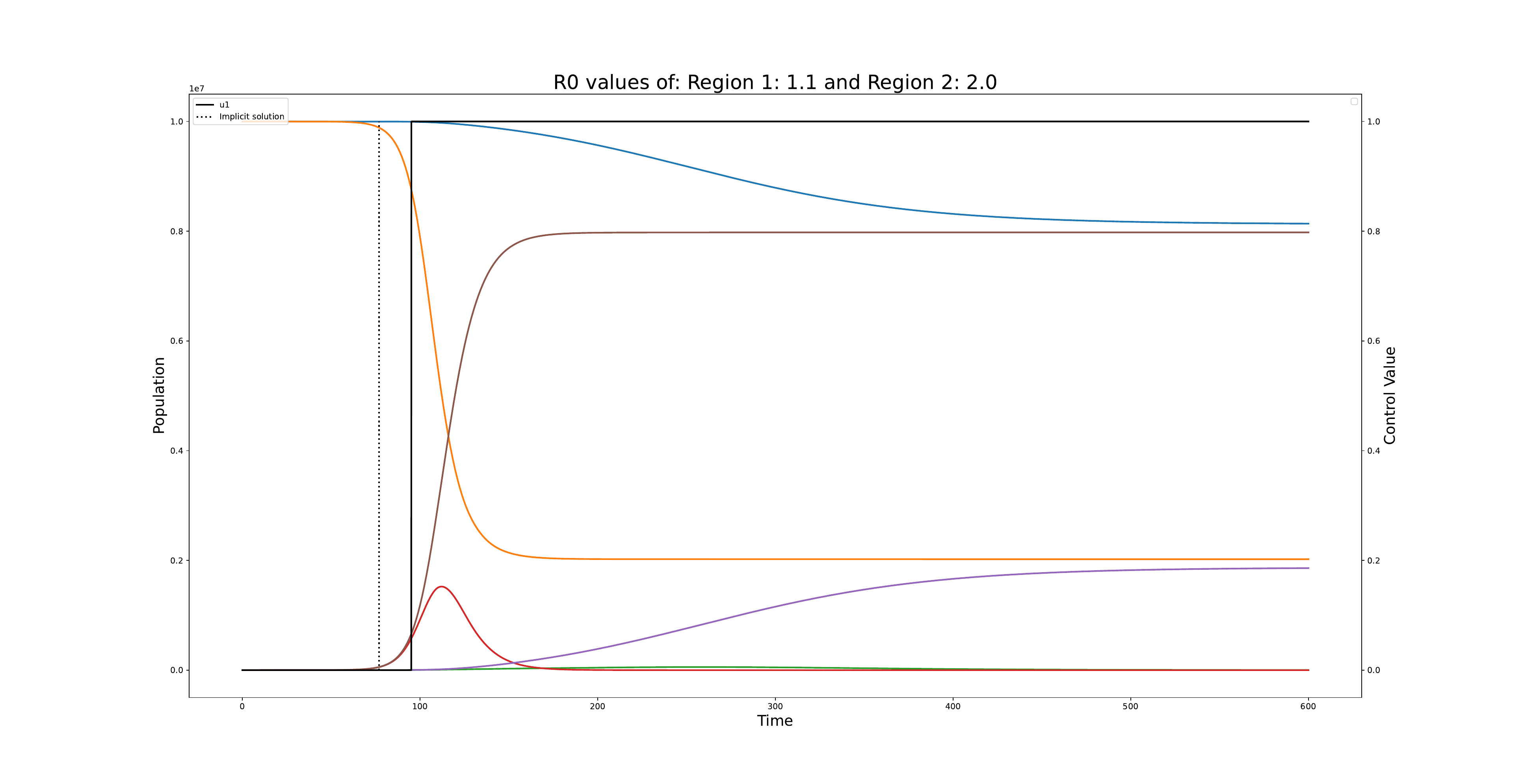}
    \caption{Output of the model simulation. The left-axis gives the region populations over time; Region 1 susceptible, infected, recovered in blue, green, purple; Region 2 susceptible, infected, recovered in orange, red, brown (respectively). The numerical optimal testing policy for Case 1 is given by the piecewise constant function in solid black; the solution for $t^*$ of Equation \eqref{OST1} is represented by the vertical dotted black line.}
    \label{fig:simulation_output}
\end{figure}

%

\subsection{Model Behavior}
This section presents both qualitative and quantitative descriptions of the model. The model behavior depends on the parameters introduced in Section~\ref{sec:analysis_in_measurable_control_setting}. We analyze the system under the following setting. The disease originates in Region~2 (through a nonzero initialization of the \(I_2\) population), while Region~1 has a lower basic reproduction number (\(R_0\)) than Region~2, representing a country able to mitigate disease spread through additional measures such as social distancing. The populations of the two regions are of comparable magnitude, and the simulation time horizon is approximately 600~days. 
Given the large number of parameters and the model’s sensitivity to their values, we report results specifically for this representative parameter set. Additional details are provided regarding the parameters that influence the switching times of the control variables.\\
We are interested in examining the optimal control policy for each scenario under the above assumptions on the model parameterization. To begin, the numerical optimum of the control policies is computed and compared with the uncontrolled simulations to assess whether the controls produce a meaningful reduction in the overall cost. For Cases~1 and~2, the cost functional in \eqref{cost} is treated as a function of the switching time in the testing policy. A scalar minimization routine from SciPy is then applied over a bounded interval around the “wave’’ of the pandemic to determine the optimal switching time. 
Figure~\ref{fig:combined_cost_profiles_case_0} illustrates the typical shape of the cost profile for Case~1, highlighting how the location of the local minimum depends on the assumed death rate in the simulation. We restrict the search to local minima due to the structure of the control functions. In the Case~1 scenario, the cost profile is observed to be monotonically decreasing after the first infection wave (see Figure~\ref{fig:combined_cost_profiles_case_0}).A possible explanation is that, since the control in Case~1 is activated once at a specific time and remains on thereafter, activating the control after the primary infection wave has little impact on disease dynamics. Consequently, a later activation leads to a lower overall cost, as the testing-related costs (see the integral term in Equation~\ref{cost}) dominate once the infection has subsided. 
A similar pattern is observed for the Case~2 scenario. However, in this case, the cost is monotonically \emph{increasing} after the main infection wave. This behavior arises because the control is active from the beginning and is switched off at \(t^{*}\). When \(t^{*}\) occurs after the peak of infections, extending the control period only accumulates additional testing costs without significantly influencing the epidemic dynamics.

\begin{figure}[!htb]
    \centering
    \includegraphics[width=0.9\linewidth]{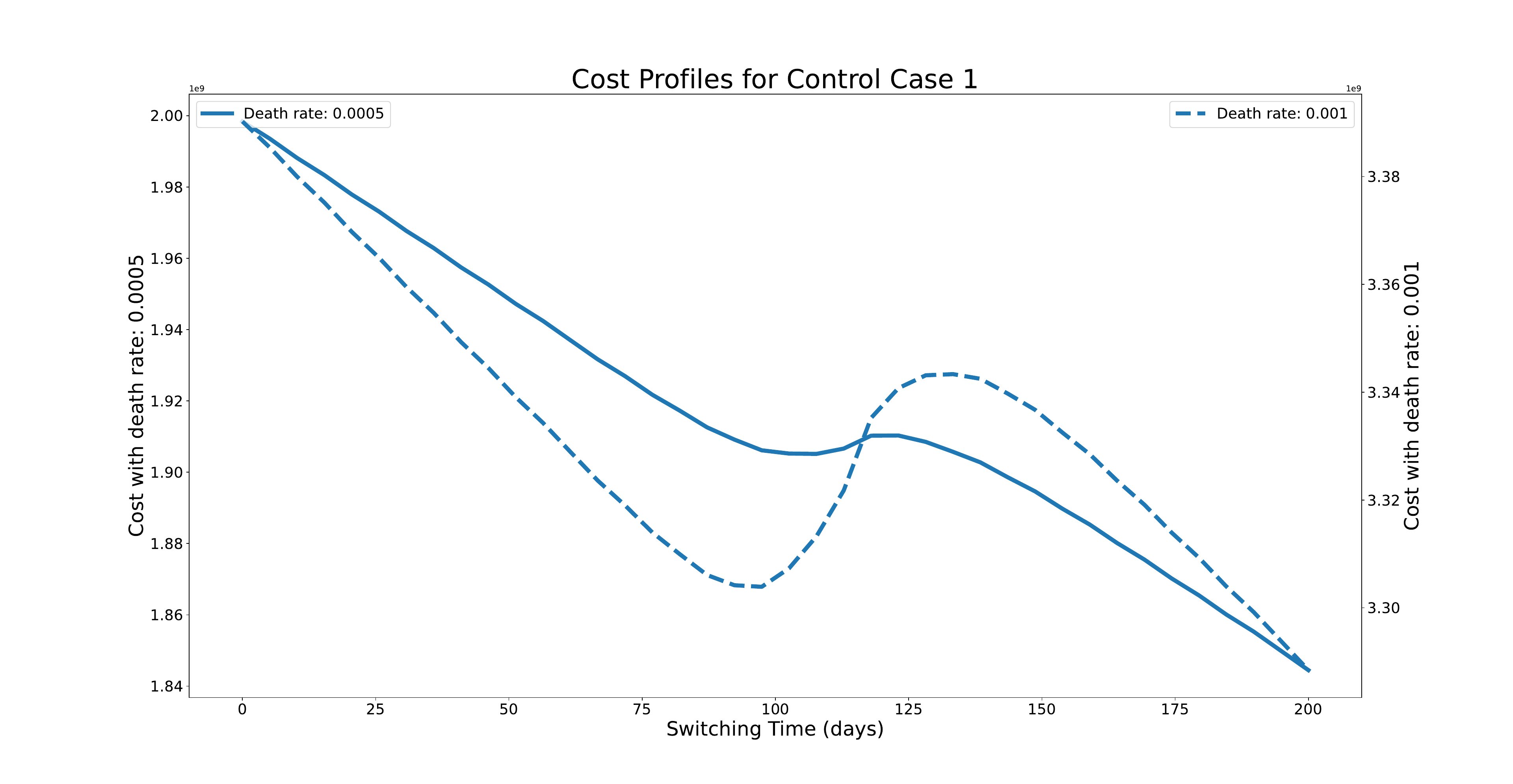}
    \caption{Cost profile for Case 1 control policy with death rates of $\frac{1}{2000}$ and $\frac{1}{1000}$.}
    \label{fig:combined_cost_profiles_case_0}
\end{figure}

In order to qualitatively investigate Cases 1 and 2 of the testing policies, the cost functional is visualized as a function of the switching time for each control. This is done similarly to the numerical optimization. However, instead of minimizing this function, it is evaluated at many points in the switching-time domain to generate a plot, and we refer to this plot as a ``cost profile.'' Note the local minima of the cost profile, which roughly align with the main wave of the infection through Region 2. Figure \ref{fig:combined_cost_profiles_case_0} demonstrates the dependence of the local minima of the cost functional on the overall death rate for the simulation. As the death rate increases, it becomes more beneficial to activate the testing policy (the local minima are lower for the higher death rate, as in Figure \ref{fig:combined_cost_profiles_case_0}, dotted line). Additionally, as the death rate increases, the numerical optimal switching time shifts to the left, indicating that earlier control is more beneficial.
\begin{figure}
    \centering
    \includegraphics[width=0.9\linewidth, trim={6cm, 2cm, 5cm, 2cm}, clip]{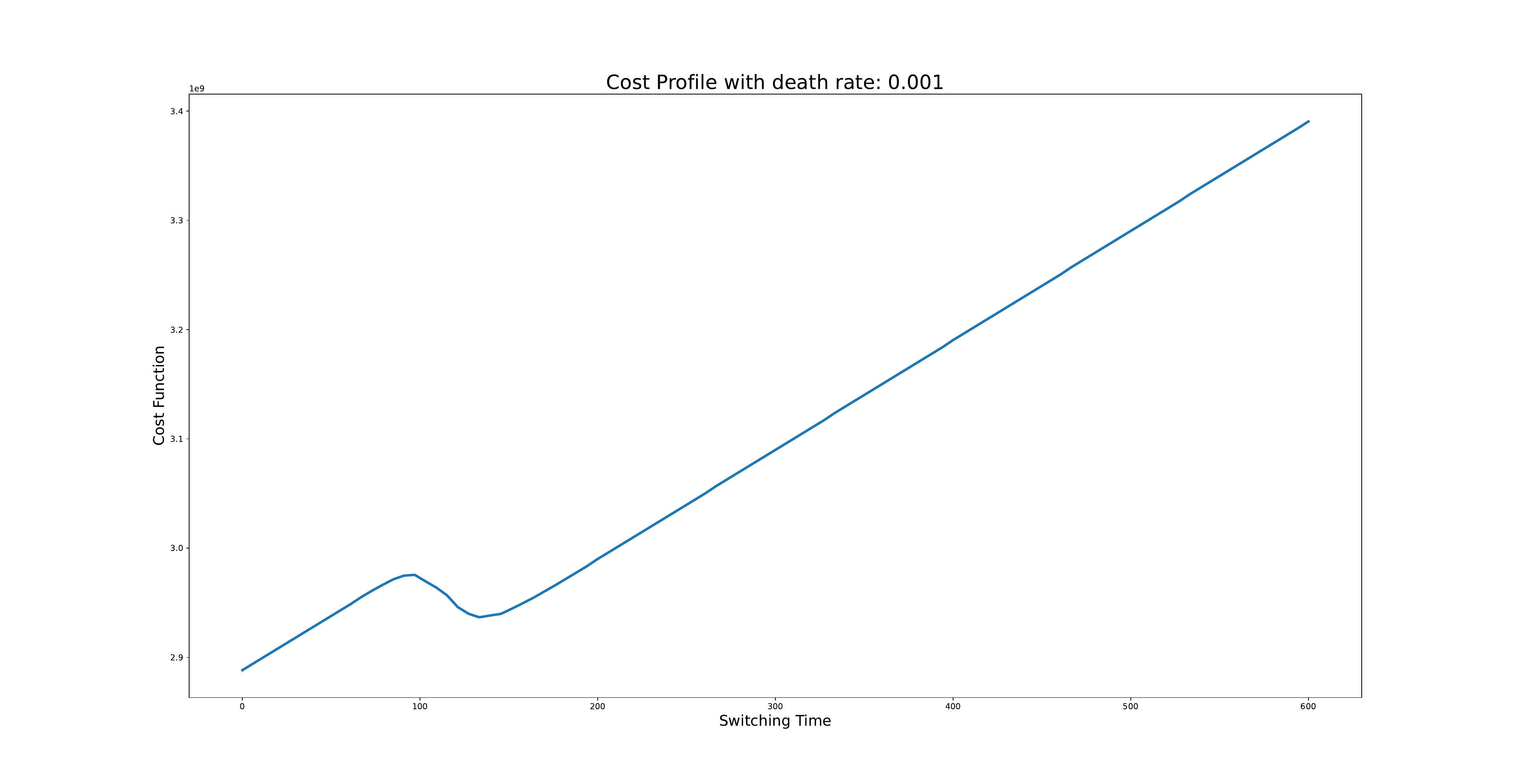}
    \caption{Cost profile for Case 2 with a moderate death rate.}
    \label{fig:cost_profile_case_1_moderate}
\end{figure}
The cost profile for Case 2 testing policies looks similar to those of Case 1, see Figure \ref{fig:cost_profile_case_1_moderate}. Again, there is a local minimum that roughly corresponds to a switching time near the wave of infection within region 2.
The total cost for the numerically optimal testing policies is reported in Table \ref{tab:testing_policy_costs}, noting that Case 3's control policy (this policy begins with no restriction on travel, then has a period of restriction followed by returning to a period of no restriction) minimizes the cost during the simulation.
\begin{table}[!htb]
    \centering
    \begin{tabular}{|c|c|c|}
        \hline
        Testing Policy & Switching Times & Optimal Cost \\
        \hline \hline
        Case 1 & 87.7 & 6092476832 \\
        \hline
        Case 2 & 148.2 & 5729567942 \\
        \hline
        Case 3 & 80.7, 133.0 & 5653997519 \\
        \hline
    \end{tabular}
    \caption{Table of costs for each testing policy case}
    \label{tab:testing_policy_costs}
\end{table}

\bibliographystyle{plain}
\bibliography{ref.bib}
\end{document}